\documentclass[]{amsart}

\usepackage{graphicx}
\usepackage{dynkin-diagrams}
\usepackage{physics}
\usepackage{hyperref} 
\hypersetup{
	colorlinks=true,
	linkcolor=blue,
	hypertexnames=false,
	filecolor=black,      
	urlcolor=blue,
	citecolor=blue
}
\pdfmapfile{=OrnementsADF.map}
\usepackage{mathtools}

\usepackage[nameinlink]{cleveref}

\theoremstyle{plain}
\newtheorem{theorem}{Theorem}[section]

\newtheorem{lemma}[theorem]{Lemma}
\newtheorem{proposition}[theorem]{Proposition}
\newtheorem{corollary}[theorem]{Corollary}

\theoremstyle{definition}
\newtheorem{definition}[theorem]{Definition}

\newtheorem{question}[theorem]{Question}

\theoremstyle{remark}
\newtheorem{remark}[theorem]{Remark}

\DeclareMathOperator{\SL}{SL}
\DeclareMathOperator{\Sp}{Sp}

\DeclareMathOperator{\SU}{SU}

\DeclareMathOperator{\Span}{Span}

\DeclareMathOperator{\mfa}{\mathfrak{a}}
\DeclareMathOperator{\mfp}{\mathfrak{p}}
\DeclareMathOperator{\mfk}{\mathfrak{k}}
\DeclareMathOperator{\mfg}{\mathfrak{g}}

\DeclareMathOperator{\mfu}{\mathfrak{u}}

\DeclareMathOperator{\mfsp}{\mathfrak{sp}}
\DeclareMathOperator{\mfsl}{\mathfrak{sl}}

\DeclareMathOperator{\X}{\mathbb{X}}

\DeclareMathOperator{\R}{\mathbb{R}}
\DeclareMathOperator{\C}{\mathbb{C}}

\DeclareMathOperator{\HH}{\mathbb{H}}

\def\acts{\curvearrowright}
\def\F{{\mathcal{F}}}

\newcommand\scalemath[2]{\scalebox{#1}{\mbox{\ensuremath{\displaystyle #2}}}}

\newcommand{\diag}[1]{\mathrm{Diag}(#1)}
\newcommand{\superdiag}[2]{\mathrm{Diag}^#1(#2)}
\DeclareMathOperator{\Iso}{Iso}

\title{Restrictions on Anosov subgroups of $\Sp(2n,\R)$}
\author{Subhadip Dey}
\address{Department of Mathematics,
	Yale University, 
	219 Prospect St, New Haven, CT 06511
}
\email{subhadip.dey@yale.edu}

\author{Zachary Greenberg}
\address{Max Planck Institute for Mathematics in the Sciences in Leipzig,
	Inselstraße 22,
	04103 Leipzig,
	Germany
}
\email{greenberg@mis.mpg.de}

\author{J. Maxwell Riestenberg}
\address{
	Max Planck Institute for Mathematics in the Sciences in Leipzig,
	Inselstraße 22,
	04103 Leipzig,
	Germany
}
\email{riestenberg@mis.mpg.de}

\thanks{
	Z. Greenberg is funded by the Deutsche Forschungsgemeinschaft (DFG, German Research Foundation) under both Project-ID 281071066-TRR 191 and Project ID 338644254-SPP2026  as well as PosLieRep ERC 101018839. 
	J.M. Riestenberg is supported by the RTG 2229 ``Asymptotic Invariants and Limits of Groups and Spaces" and by the DFG under Project-ID 338644254 - SPP2026. 
	Z. Greenberg and J.M. Riestenberg completed this work at the University of Heidelberg}

\begin{document}
	
	\begin{abstract}
		Let $n\in\mathbb{N}$ and let $\Theta \subset \{1,\dots,n\}$ be a non-empty subset.
		We prove that if $\Theta$ contains an odd integer, then any $P_\Theta$-Anosov subgroup of ${\rm Sp}(2n,\mathbb{R})$ is virtually isomorphic to a free group or a surface group. 
		In particular, any Borel Anosov subgroup of ${\rm Sp}(2n,\mathbb{R})$ is virtually isomorphic to a free or surface group.
		On the other hand, if $\Theta$ does not contain any odd integers, then there exists a $P_\Theta$-Anosov subgroup of ${\rm Sp}(2n,\mathbb{R})$ which is not virtually isomorphic to a free or surface group.
		We also exhibit new examples of maximally antipodal subsets of certain flag manifolds; these arise as limit sets of rank $1$ subgroups.
	\end{abstract}
	
	\maketitle
	
	\section{Introduction}
	Since their introduction by Labourie \cite{Lab06}, Anosov subgroups of semisimple Lie groups have come to be regarded as the right generalization of convex cocompact actions on hyperbolic spaces to higher rank. An Anosov subgroup $\Gamma$ of a semisimple Lie group $G$ is word-hyperbolic and comes equipped with a boundary map from the Gromov boundary of $\Gamma$ to a flag manifold $\F=G/P$ \cite{GW12}. Alternatively, they can be characterized in terms of their coarse geometry as those subgroups with uniformly regular undistorted orbits in the symmetric space $\X=G/K$ \cite{KLP17}.
	
	It is intriguing to ask how large is the class of word-hyperbolic groups that appear as Anosov subgroups of semisimple Lie groups.
	Danciger-Gu\'eritaud-Kassel \cite{DGK18} showed that a large family of word-hyperbolic groups, namely, any hyperbolic right-angled Coxeter group admit Anosov representations.
	However, Kapovich \cite{KapPolygon} gave examples of infinite hyperbolic groups whose linear representations always have finite image and, as a consequence, which cannot be realized as Anosov subgroups. 
	Apart from linearity, no other obstructions for hyperbolic groups are currently known which prohibit them from admitting Anosov representations.
	Cf. \cite[Problem 50.1]{canary2021anosov}.
	
	On the other hand, fixing a semisimple Lie group $G$ and a set of simple roots $\Theta$, one may seek to understand obstructions on hyperbolic groups admitting $\Theta$-Anosov representations into $G$.
	The work of Canary-Tsouvalas \cite{CT20} gave an upper bound on the cohomological dimension of $\{k\}$-Anosov subgroups of $\SL(n,\R)$.
	In this regard, a challenging question was posed by Andr\'{e}s Sambarino (see \cite[\S7]{CT20}): 
	\begin{question}[Sambarino]
		Is every Borel Anosov subgroup of $\SL(n,\R)$ virtually isomorphic to a free or surface group?
	\end{question}
	While this question remains open in general, progress has been made by Canary-Tsouvalas \cite{CT20}, Tsouvalas \cite{Tso20}, and the first author \cite{dey22}.
	
	In this paper, we address this question for $G = \Sp(2n,\R)$ and arbitrary $\Theta$. 
	Before stating our main result, we set up a notation for the simple roots:
	The root system for the symplectic Lie algebra $\mathfrak{sp}(2n,\R)$ is of type $\mathrm{C}_n$; it has $n$ simple roots and we denote them by natural numbers $1,\dots,n$ in the order as they appear in the Dynkin diagram of type $\mathrm{C}_n$, with $n$ denoting the unique long root.
	See \Cref{fig:dynkin diagram Cn}.
	
	\begin{figure}[h]
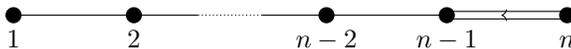

		\centering
		\begin{equation*}
			\begin{dynkinDiagram}[text style/.style={scale=1},
				edge length=.8cm,
				scale=2,
				labels={1,2,n-2,n-1,n},
				label macro/.code={{{#1}}}
				]C{}
			\end{dynkinDiagram}
		\end{equation*}
		\caption{Dynkin diagram of type ${\rm C}_n$}
		\label{fig:dynkin diagram Cn}
	\end{figure}
	
	Our main result is as follows.
	
	\begin{theorem}\label{thm:main}
		Let $n\in\mathbb{N}$ and let $\Theta \subset \{1,\dots,n\}$ be a non-empty subset.
		\renewcommand{\theenumi}{(\roman{enumi})}
		\begin{enumerate}
			\item If $\Theta$ contains an odd integer, then any $\Theta$-Anosov subgroup of ${\rm Sp}(2n,\mathbb{R})$ is virtually isomorphic to a free group or a surface group.
			\item If $\Theta$ does not contain any odd integers, then the fundamental group of any closed, orientable, hyperbolic 3-manifold may be realized as a $\Theta$-Anosov subgroup of ${\rm Sp}(2n,\mathbb{R})$.
		\end{enumerate}
	\end{theorem}
	
	A special case of the first statement is that every Borel Anosov\footnote{I.e., $\Theta$-Anosov, where $\Theta = \{1,\dots,n\}$.} subgroup of $\Sp(2n,\R)$, $n\in\mathbb{N}$, is virtually isomorphic to a free or surface group.
	A Borel Anosov subgroup of $\Sp(2n,\R)$ can be viewed as a Borel Anosov subgroup of $\SL(2n,\R)$ via the natural inclusion \[\Sp(2n,\R)\hookrightarrow \SL(2n,\R),\] so this special case also follows from \cite{CT20} when $n=2$ and from \cite{dey22} when $n \not\equiv 0$ mod $4$.
	More generally, the above inclusion takes $\Theta$-Anosov subgroups of $\Sp(2n,\R)$ to $\Theta'$-Anosov subgroups of $\SL(2n,\R)$, where $\Theta'$ is the subset of $\{1,\dots,2n-1\}$ consisting of $\Theta$ and $\{2n-k :\ k \in \Theta\}$. 
	In particular, when $n$ is odd and $\Theta$ contains $\{n\}$, the result follows from \cite{Tso20}. 
	After circulating an early draft of this paper, we learned that \Cref{thm:main}(i) was independently obtained by Beatrice Pozzetti and Kostas Tsouvalas with different techniques \cite{PT23}.
	
	\medskip
	
	The restrictions on Anosov subgroups obtained here and in \cite{dey22} are based on criteria for antipodal circles in partial flag manifolds to be {\em maximally antipodal}. 
	It turns out that any $\Theta$-Anosov subgroup of $G$ with a maximally antipodal limit set cannot appear as an infinite index subgroup of another $\Theta$-Anosov subgroup of $G$, see \Cref{prop:maximal_anosov}.
	We produce examples of higher-dimensional spheres which are maximally antipodal in certain partial flag manifolds. 
	We let $\mathcal{F}_{2,2n-2}$ denote the partial flag manifold consisting of pairs $(V,W)$ of vector subspaces of $\R^{2n}$ with $\dim(V)=2$, $\dim(W)=2n-2$ and $V\subset W$.
	We let $\Iso_2(\R^{2n})$ denote the space of isotropic $2$-planes in $\R^{2n}$ with respect to a symplectic form.
	
	\begin{theorem}\label{thm:main2}\leavevmode
		\renewcommand{\theenumi}{(\roman{enumi})}
		\begin{enumerate}
			\item There exists a maximally antipodal subset of $\mathcal{F}_{2,2n-2}$ homeomorphic to $S^2$.
			\item There exists a maximally antipodal subset of $\Iso_2(\R^{2n})$ homeomorphic to $S^{2n-3}$.
		\end{enumerate}
	\end{theorem}
	
	These spheres arise as limit sets of totally geodesic copies of $\HH^3$ and $\mathbb{CH}^{n-1}$ inside the symmetric space $\mathbb{X} = \Sp(2n,\R)/U(n)$. 
	In particular, we note that the partial flag manifold of isotropic $2$-planes $\Iso_2(\R^6,\omega)$ admits maximally antipodal subsets homeomorphic to $S^2$ and $S^3$.
	
	\medskip
	
	\subsection*{Outline of the paper}
	
	The proof of \Cref{thm:main}(i) is based on extending the techniques of \cite{dey22} to partial flag manifolds in general, which we do in \S\ref{sec:FlagManifolds}, supplemented with an elementary observation about partial flag manifolds associated to $\Sp(2n,\R)$; see \S\ref{sec:res_symplectic}. 
	
	The proof of \Cref{thm:main}(ii), which we discuss in \S\ref{sec:examples}, is constructive: We show that for all $n \ge 2$, there exists a representation $\rho_n \colon \SL(2,\C) \to \Sp(2n,\R)$, which geometrically corresponds to a certain totally geodesic embedding of $\HH^3$ in the symmetric space $\mathbb{X} = \Sp(2n,\R)/U(n)$, such that $\rho_n$ maps any convex cocompact subgroup of $\SL(2,\C)$ to a $\Theta_{\rm even}$-Anosov subgroup of $\Sp(2n,\R)$; see \Cref{thm:SL2Cexample}.
	Since the fundamental group of any closed, orientable, hyperbolic 3-manifold can be realized as a uniform lattice in $\SL(2,\C)$ (see \cite{culler1986lifting}), \Cref{thm:main}(ii) follows as a special case.
	
	The copies of $S^2$ that we mention in \Cref{thm:main2}(i) arise as flag limit sets of $\rho_n(\SL(2,\C))$. 
	We prove they are maximally antipodal in \Cref{thm:sl2c limit set is maximally antipodal}. 
	A similar construction applied to the inclusion $\SU(n-1,1) \subset \Sp(2n,\R)$ produces the $S^{2n-3}$ in \Cref{thm:main2}(ii). 
	We prove these are maximally antipodal in \Cref{thm:MaximalAntipodalSU}.
	
	\subsubsection*{Acknowledgements} 
	
	We are grateful to Anna Wienhard, Fanny Kassel and Kostas Tsouvalas for interesting conversations related to the results of this paper.
	We would also like to thank Beatrice Pozzetti for a useful suggestion which simplified and improved \Cref{thm:main}(ii). 
	We thank the referee for valuable feedback which improved the exposition of the paper.
	
	\section{Preliminaries}\label{sec:FlagManifolds}
	
	Let $G$ be a connected real semisimple Lie group with a finite center.
	We fix a Cartan decomposition of the Lie algebra $\mathfrak{g}$ of $G$,
	\[
	\mathfrak{g} = \mathfrak{k} \oplus \mathfrak{p},
	\]
	and then we fix a maximal abelian subspace $\mathfrak{a}\subset \mathfrak{p}$.
	Let $\Sigma$ denote the set of {\em roots}, i.e., the set of all nonzero elements $\alpha\in\mathfrak{a}^*$ for which the associated weight space
	$
	\mathfrak{g}_\alpha = \{ Y\in \mathfrak{g} :\ ({\rm ad}\,A) Y = \alpha(A)Y, \; \forall A\in\mathfrak{a} \}
	$
	is nonzero.
	The adjoint action of $\mathfrak{a}$ on $\mathfrak{g}$ is a commuting family of diagonalizable linear transformations and hence we obtain a simultaneous eigenspace decomposition, called the {\em (restricted) root space decomposition} of $\mathfrak{g}$:
	\[
	\mathfrak{g} = \mathfrak{g}_0 \oplus \bigoplus_{\alpha\in\Sigma} \mathfrak{g}_\alpha.
	\]
	We choose a set of {\em positive roots} $\Sigma^+ \subset \Sigma$, which (equivalently) corresponds to a choice of a
	(closed) {\em positive Weyl chamber} $\overline{\mathfrak{a}^+} = \{A\in \mathfrak{a} :\ \alpha(A) \ge 0,\; \forall \alpha\in\Sigma^+\}$.
	The set of {\em simple roots} is denoted by $\Delta \subset \Sigma^+$. 
	The {\em (restricted) Weyl group} is the group generated by reflections in the roots.
	It contains a unique element $w_0$ taking $\mfa^+$ to $-\mfa^+$. 
	The {\em opposition involution} $-w_0: \mfa \to \mfa$ induces an operation on the simple roots 
	\[
	{\rm i} \colon \Delta \to \Delta.
	\]
	that we also call the opposition involution. 
	Note that for $G = \Sp(2n,\R)$, which is the main focus of this article, ${\rm i} \colon \Delta \to \Delta$ is the identity map. 
	
	Every (nonempty) subset $\Theta \subset \Delta$ determines a pair of nilpotent subalgebras
	\[
	\mathfrak{u}_\Theta \coloneqq \sum_{\alpha\in \Sigma^+_\Theta} \mathfrak{g}_\alpha \quad\text{and}\quad \mathfrak{u}_\Theta^{\rm opp} \coloneqq \sum_{\alpha\in \Sigma^+_\Theta} \mathfrak{g}_{-\alpha},
	\]
	where $\Sigma^+_\Theta \coloneqq \Sigma^+\setminus {\rm Span}(\Delta\setminus \Theta)$.
	The normalizer of $\mathfrak{u}_{\Theta}$ in $G$ for the adjoint action $G \acts \mathfrak{g}$ is called the {\em standard\footnote{Here ``standard'' means standard with respect to the above choices.} parabolic subgroup}, denoted by $P_\Theta$.
	Similarly, normalizer of $\mathfrak{u}^{\rm opp}_{\Theta}$ in $G$ is the standard {\em opposite} parabolic subgroup, which is denoted by $P_\Theta^{\rm opp}$.
	The closed subgroup $U_\Theta = \exp (\mathfrak{u}_\Theta)$ of $P_\Theta$ is called the {\em unipotent radical} of $P_\Theta$.
	
	Given a nonempty subset $\Theta \subset \Delta$, the corresponding {\em flag manifold}
	\[
	\F_\Theta = G/P_\Theta
	\]
	is a $G$-homogeneous space.
	In this paper, we only consider those standard parabolic subgroups $P_\Theta$ which are conjugate to its opposite parabolic subgroup $P_\Theta^{\rm opp}$; equivalently, $\Theta = {\rm i}(\Theta)$.
	In $ \F_\Theta $, the unique fixed point of $P_\Theta$ (resp. $P_\Theta^{\rm opp}$) will be denoted by
	$\tau_\Theta$ (resp. $\tau_\Theta^{\rm opp}$).
	The action $P_\Theta\acts \F_\Theta$ has a unique open orbit $C(\tau_\Theta) \coloneqq P_\Theta \cdot \tau_\Theta^{\rm opp}$.
	For any point $\tau \in C(\tau_\Theta)$,  the unipotent radical $U_\Theta$ of $P_\Theta$ yields  parametrization $U_\Theta \to C(\tau_\Theta)$, given by
	$u \mapsto u\cdot \tau$.
	
	A pair of points $\tau_\pm \in\F_\Theta$ is called {\em antipodal} (or {\em transverse}) if there exists $g\in G$ such that $g \tau_- = \tau_\Theta$ and $g \tau_+ \in C(\tau_\Theta)$.
	For $\tau\in\F_\Theta$, the set of all points in $\F_\Theta$ antipodal to $\tau$ is denoted by $C(\tau)$.
	Clearly, if $\tau_\pm \in\F_\Theta$ is a pair of antipodal points, then $g\tau_\pm$
	is too, for any $g\in G$.
	Moreover, for all $\tau\in\F_\Theta$ and all $g\in G$,
	$
	g\cdot C(\tau) = C(g\tau).
	$
	
	\subsection{The inversion property}
	Let $\Theta$ be an ${\rm i}$-invariant subset of $\Delta$.
	
	\begin{definition}[Inversion map]\label{def:inversion map}\leavevmode
		The {\em inversion map}  is the involution
		\[\iota : C(\tau_\Theta) \to C(\tau_\Theta), \quad
		\iota(\tau) = u^{-1}_{\tau} \tau_\Theta^{\rm opp},
		\]
		where $u_\tau\in U_\Theta$ is the unique element such that $\tau = u_\tau\tau_\Theta^{\rm opp}$.
	\end{definition}
	
	\begin{lemma}\label{lem:inversion_preserves}
		The inversion map $\iota : C(\tau_\Theta) \to C(\tau_\Theta)$
		preserves $C(\tau_\Theta) \cap C(\tau_\Theta^{\rm opp})$.
	\end{lemma}
	
	\begin{proof}
		For $\tau\in C(\tau_\Theta) \cap C(\tau_\Theta^{\rm opp})$, let $u_\tau\in U_\Theta$ be the unique element such that $\tau = u_\tau \tau_\Theta^{\rm opp}$.
		Then, $u_\tau^{-1}\cdot \{\tau,\tau_\Theta^{\rm opp}\} = \{\tau_\Theta^{\rm opp},u_\tau^{-1}\tau_\Theta^{\rm opp}\}$.
		Since $\tau$ and $\tau_\Theta^{\rm opp}$ are antipodal and the action $U_\Theta\acts \F_\Theta$ preserves the property of being antipodal, it follows that $\tau_\Theta^{\rm opp}$ and $u_\tau^{-1}\tau_\Theta^{\rm opp}$ are antipodal.
		Thus, $\iota(\tau) = u_\tau^{-1}\tau_\Theta^{\rm opp}\in C(\tau_\Theta) \cap C(\tau_\Theta^{\rm opp})$.
	\end{proof}
	
	By the \Cref{lem:inversion_preserves}, the map $\iota$ induces a well-defined involution on the set of connected components of $C(\tau_\Theta) \cap C(\tau_\Theta^{\rm opp})$.
	
	\begin{definition}[Inversion property]
		The flag manifold $\F_\Theta$ is said to have {\em Property (I)} if the inversion map $\iota: C(\tau_\Theta) \to C(\tau_\Theta) $ does not leave invariant any connected components of $C(\tau_\Theta) \cap C(\tau_\Theta^{\rm opp})$.
	\end{definition}
	
	We ask the following question:
	
	\begin{question}
		Which flag manifolds $\F_\Theta$ of $G$ have Property (I)?
	\end{question}
	
	The main motivation for the above question is that an affirmative answer would imply strong restrictions on hyperbolic groups admitting $P_\Theta$-Anosov representations into $G$; see \S\ref{sec:AnosocRestriction}.
	In \cite[Thm. 2.4]{dey22}, it is shown that, for all natural numbers $n\not\equiv 0,\pm1$ mod $8$, the full flag manifold $\F_\Delta$ of $G = \SL(n,\R)$ has Property (I).
	In the present paper, we show that $\Sp(2n,\R)/P_\Theta$ has Property (I) if and only if $\Theta$ contains an odd integer; one implication is \Cref{lem:symplectc_propI} and the other follows from \Cref{thm:Anosovrestriction,thm:SL2Cexample}. 
	
	The following result shows that Property (I) is ``increasing,'' which could be helpful to study the question above: 
	Let $\Theta, \Theta'\subset \Delta$ be ${\rm i}$-invariant subsets such that $\Theta'\subset \Theta$.
	The corresponding standard parabolic subgroups
	$P_\Theta$, $P_{\Theta'}$ satisfy $P_\Theta < P_{\Theta'}$ and, therefore, we have a well-defined $G$-equivariant surjective morphism 
	\[
	\pi: \F_{\Theta} \to \F_{\Theta'},
	\]
	whose fiber over any point $\tau'\in P_{\Theta'}$ is isomorphic to $P_{\Theta'}/P_\Theta$.
	
	\begin{proposition}\label{prop:increasing}
		If $\F_{\Theta'}$ has Property (I), then so does $\F_\Theta$.
	\end{proposition}
	
	\begin{proof}
		Since $\Theta'\subset \Theta$, we have $\mathfrak{u}_{\Theta'} \subset \mathfrak{u}_\Theta$.
		Therefore, the unipotent radical $U_{\Theta'} = \exp(\mathfrak{u}_{\Theta'})$ of $P_{\Theta'}$ is a subgroup of the unipotent radical $U_{\Theta} = \exp(\mathfrak{u}_{\Theta})$ of $P_{\Theta}$.
		
		We first claim that $ \pi(C(\tau_\Theta)) = C(\tau_{\Theta'}) $:
		Indeed, since $\tau_{\Theta'} = \pi(\tau_\Theta)$ and $\tau_{\Theta'}^{\rm opp} = \pi(\tau_\Theta^{\rm opp})$,
		we get
		\[
		\pi(C(\tau_\Theta)) = \pi(P_\Theta\cdot \tau_\Theta^{\rm opp}) = 
		P_\Theta\cdot\pi( \tau_\Theta^{\rm opp}) \subset P_{\Theta'}\cdot \tau_{\Theta'}^{\rm opp}
		= C(\tau_{\Theta'})
		\]
		and, on the other hand,
		\[
		\pi(C(\tau_\Theta)) = \pi(U_\Theta\cdot \tau_\Theta^{\rm opp}) = 
		U_\Theta\cdot\pi( \tau_\Theta^{\rm opp}) \supset U_{\Theta'}\cdot \tau_{\Theta'}^{\rm opp}
		= C(\tau_{\Theta'}),
		\]
		proving the desired equality.
		In particular,   
		\begin{equation}\label{eqn:prop:increasing}
			\pi(C(\tau_\Theta) \cap C(\tau_\Theta^{\rm opp}) ) 
			\subset C(\tau_{\Theta'}) \cap C(\tau_{\Theta'}^{\rm opp}) 
		\end{equation}
		
		We next claim that $U_1$, the stabilizer of $ \tau_{\Theta'}^{\rm opp}$ in $U_\Theta$, is path-connected:  Indeed, it follows by considering the long exact sequence of homotopy groups corresponding to the $U_{\Theta}$-equivariant fibration 
		\[
		\pi : C(\tau_\Theta) \to C(\tau_{\Theta'})
		\]
		that $\pi_0(F,\tau_{\Theta}^{\rm opp})$ is singleton, where $F = \pi^{-1}(\tau_{\Theta'}^{\rm opp}) = U_1\cdot \tau_{\Theta}^{\rm opp}$.
		Thus, $F$ and, hence, $U_1$ are path-connected.
		Consequently, since $U_1$ is path-connected and stabilizes $\tau_{\Theta'}$ and $\tau_{\Theta'}^{\rm opp}$, it follows that $U_1$ preserves the connected components of $C(\tau_{\Theta'})\cap C(\tau_{\Theta'}^{\rm opp})$.
		
		Now we finish the proof:
		Let $\tau\in C(\tau_{\Theta})\cap C(\tau_{\Theta}^{\rm opp})$ be any point. By definition, 
		$
		\iota(\tau) = u^{-1}\tau_{\Theta}^{\rm opp},
		$
		where $u\in U_\Theta$ is the unique element such that $\tau = u\tau_{\Theta}^{\rm opp}$.
		Moreover, let $\tau' \coloneqq \pi(\tau) = \pi(u\tau_{\Theta}^{\rm opp}) = u\tau_{\Theta'}^{\rm opp}$.
		Let $u'\in U_{\Theta'}$ be the unique element such that $\tau' = u'\tau_{\Theta'}^{\rm opp}$. Then, $u^{-1} u'\in U_{\Theta}$ stabilizes $\tau_{\Theta'}^{\rm opp}$ or, equivalently,
		\begin{equation}\label{eqn:prop:increasing:two}
			u^{-1}\in U_1(u')^{-1}.
		\end{equation}
		Since $\F_{\Theta'}$ has Property (I), $\tau'$ and $\iota(\tau') = (u')^{-1}\tau_{\Theta'}^{\rm opp}$ lie in different connected components of $C(\tau_{\Theta'})\cap C(\tau_{\Theta'}^{\rm opp})$.
		Furthermore,  since $U_1$ is path-connected, it follows that ${\tau'}$ and $U_1\cdot \iota(\tau') = U_1(u')^{-1}\tau_{\Theta'}^{\rm opp}$
		also lie in different connected components of $C(\tau_{\Theta'})\cap C(\tau_{\Theta'}^{\rm opp})$.
		By \eqref{eqn:prop:increasing:two}, $\pi(\iota(\tau)) = u^{-1}\tau_{\Theta'}^{\rm opp}\in U_1(u')^{-1}\tau_{\Theta'}^{\rm opp}$.
		Thus, $\pi(\tau)$ and $\pi(\iota\tau)$ lie in different connected components of $C(\tau_{\Theta'})\cap C(\tau_{\Theta'}^{\rm opp})$.
		By \eqref{eqn:prop:increasing}, it then follows that
		$\tau$ and $\iota(\tau)$ lie in distinct connected components of 
		$C(\tau_\Theta) \cap C(\tau_\Theta^{\rm opp}) $.
		This completes the proof.
	\end{proof}
	
	\begin{corollary}
		If the full flag manifold $\F_\Delta$ of $G$ does not have Property (I), 
		then no other flag manifold $\F_\Theta$ of $G$, where $\Theta\subset\Delta$ is ${\rm i}$-invariant, has Property (I).
	\end{corollary}
	
	For example, it is shown in \cite[\S2.5]{dey22} that Property (I) fails for the full flag manifold of $\SL(n,\R)$, for all $n \equiv \pm1$ mod $8$.
	
	\subsection{Restrictions on Anosov subgroups}\label{sec:AnosocRestriction}
	
	In this subsection, we show that Property (I) restricts the group theoretic structure of Anosov subgroups.
	See \Cref{thm:Anosovrestriction} below for a precise statement.
	
	We recall a few definitions: Let $\Theta\subset \Delta$ be an ${\rm i}$-invariant subset. A subset $\Lambda$ of $\F_\Theta = G/P_\Theta$ is called {\em antipodal} if every pair of distinct points in $\Lambda$ is antipodal.
	An antipodal subset $\Lambda \subset \F_\Theta$ is called {\em maximally antipodal} if for all $\tau \in \F_\Theta$, there exists $\lambda\in\Lambda$ such that $\tau$ is not antipodal to $\lambda$.
	More generally, an antipodal subset $\Lambda \subset \F_\Theta$ is called {\em locally maximally antipodal}
	if there exists a neighborhood $N$ of $\Lambda$ such that, for all $\tau\in N$, there exists $\lambda\in c$ such that $\tau$ and $\lambda$ are not antipodal.
	
	\begin{lemma}\label{lem:propertyIimpliesmaximallyantipodal}
		If $\F_\Theta$ has Property (I), then all antipodal circles in $\F_\Theta$ are locally maximally antipodal.
	\end{lemma}
	
	In the statement above, by an ``antipodal circle'' in $\F_\Theta$, we mean the image of an embedding $\phi: S^1\to\F_\Theta$ such that $\phi(S^1)$ is an antipodal subset of $\F_\Theta$.
	
	\begin{proof}[Proof of Lemma \ref*{lem:propertyIimpliesmaximallyantipodal}]
		Let $\tau_\pm\in\F_\Theta$ be any pair of antipodal points, and let $f:[-1,1] \to \F_\Theta$, $f(\pm 1) = \tau_\pm$, be a continuous map.
		If $f((-1,1)) \subset C(\tau_-)\cap C(\tau_+)$, then, let $\Omega$ be the connected component of  $C(\tau_-)\cap C(\tau_+)$ containing the image $f((-1,1))$.
		Since $\F_\Theta$ has Property (I), by the same argument used in the proof of \cite[Theorem A]{dey22} (see \cite[\S 3]{dey22}), it follows that every point in $\Omega$ is not antipodal to some point in $f(-1,1)$.
		
		If $c$ is an antipodal circle, then, choose any distinct points $\tau_1,\tau_2,\tau_3 \in c$.
		For distinct indices $i,j,k\in\{1,2,3\}$, let $\Omega_{ijk}$ be the connected component of $C(\tau_i)\cap C(\tau_k)$ containing $\tau_j$. Then $N = \Omega_{123} \cup \Omega_{231} \cup \Omega_{312}$ is an open neighborhood of $c$. 
		By the previous paragraph, every point in $N$ is not antipodal to some point in $c$, cf. \cite[Corollary B]{dey22}.
		Hence, $c$ is locally maximally antipodal.
	\end{proof}
	
	A subgroup $\Gamma$ of $G$ is said to be {\em $\Theta$-boundary embedded} if $\Gamma$ is hyperbolic and there exists a $\Gamma$-equivariant continuous map
	\begin{equation}\label{def:BE}
		\xi : \partial_\infty\Gamma \to \F_\Theta
	\end{equation}
	from the Gromov boundary $\partial_\infty\Gamma$ of $\Gamma$ to $\F_\Theta$, which sends every pair of distinct points in $\partial_\infty\Gamma$ to a pair of antipodal points in $\F_\Theta$.
	A non-elementary boundary embedded subgroup is necessarily discrete, since $\xi$ is an embedding and the action of $\Gamma$ on $\partial_\infty \Gamma$ is a convergence group action, see Freden \cite{Freden}. 
	We refer our readers to Kapovich-Leeb-Porti \cite[\S5.2]{KLP17} for more details on $\Theta$-boundary embedded subgroups.
	
	A subgroup $\Gamma$ of $G$ is said to be ($P_\Theta$- or) {\em $\Theta$-Anosov} if it is $\Theta$-boundary embedded with a {\em strongly dynamics preserving} boundary map $\xi : \partial_\infty\Gamma \to \F_\Theta$ (as in \eqref{def:BE}); see \cite[Definition 2.10]{GW12} for a precise definition.
	In this situation, the image of this (unique) strongly dynamics preserving boundary map $\xi : \partial_\infty\Gamma \to \F_\Theta$ is called the {\em $\Theta$-limit set} of $\Gamma$.
	
	\begin{theorem}\label{thm:Anosovrestriction}
		Suppose that $\F_\Theta$ has Property (I).
		Then, any $\Theta$-boundary embedded subgroup of $G$ is virtually isomorphic to either a free group or a surface group.
		
		In particular, the same conclusion holds for any $\Theta$-Anosov subgroup of $G$.
	\end{theorem}
	
	\begin{proof}
		Let $\Gamma < G$ be a $\Theta$-boundary embedded subgroup and let $\xi : \partial_\infty \Gamma \to \F_\Theta$ be a boundary embedding as in \eqref{def:BE}. Assume that $\Gamma$  not virtually free.
		Then, using \cite[Corollary 2]{MR2146190}, we discover an embedding $j :S^1 \to \partial_\infty\Gamma$.  We aim to demonstrate that $j$ is surjective, and therefore, a homeomorphism:
		
		We observe that $\xi \circ j(S^1)\subset  \F_\Theta$ is an antipodal circle. By \Cref{lem:propertyIimpliesmaximallyantipodal}, it is locally maximally antipodal. Since $\xi : \partial_\infty \Gamma \to \F_\Theta$ maps pairwise distinct points to pairwise antipodal points, there exists an open neighborhood $N$ of $\xi (  j(S^1))$ in $\F_\Theta$ which does not contain $\xi(z)$, for any $z\in \partial_\infty \Gamma\setminus j(S^1)$.
		Consequently, $j(S^1) = \xi^{-1}(N)$ is open (and closed) in $\partial_\infty\Gamma$.
		Since hyperbolic fixed points are dense in $\partial_\infty\Gamma$, there exists an infinite order element $\gamma\in\Gamma$ such that the attractive fixed point $\gamma_+$ of $\gamma$ lies in $j(S^1)$. 
		Then, we must have $\gamma(j(S^1)) = j(S^1)$ since $\gamma(j(S^1))\cap j(S^1)$ is nonempty and $j(S^1)$ is both closed and open in $\partial_\infty\Gamma$.
		
		Notably, the repulsive fixed point $\gamma_-\in\partial_\infty\Gamma$ of $\gamma$ also lies in $j(S^1)$.  This conclusion arises because  $\gamma_-$ is the accumulation point of the sequence $(\gamma^{-n} z)$, where $z\in j(S^1)$ is any point distinct from of $\gamma_+$, and $(\gamma^{-n} z)$ remains within $j(S^1)$. 
		
		Now we can show that $j: S^1 \to \partial_\infty \Gamma$ is surjective. As $\gamma$ preserves $j(S^1)$, it also preserves its complement $\partial_\infty\Gamma \setminus j(S^1)$.
		Were $\partial_\infty\Gamma \setminus j(S^1)$ nonempty, then, for any point $z\in \partial_\infty\Gamma \setminus j(S^1)$, the sequence $(\gamma^n z)$ would accumulate at $\gamma_+$ (note that $z\ne \gamma_-$ as $\gamma_-\in j(S^1)$).
		However, since the sequence $(\gamma^n z)$ lies the closed subset $\partial_\infty\Gamma \setminus j(S^1)$, it cannot have an accumulation point outside it, leading to a contradiction! 
		
		Therefore, we have arrived at the conclusion that $\partial_\infty \Gamma$ is homeomorphic to a circle.
		Applying the deep results of Tukia \cite{Tukia}, Gabai \cite{Gabai}, Freden \cite{Freden}, Casson-Jungreis \cite{CJ} (see also \cite[Theorem 5.4]{BK}), it follows that $\Gamma$ contains a finite index subgroup isomorphic to the fundamental group of a closed hyperbolic surface.
	\end{proof}
	
	Our next result shows that $\Theta$-Anosov subgroups of $G$ with maximally antipodal $\Theta$-limit sets are {\em maximal} in the class of $\Theta$-Anosov subgroups of $G$ in the sense that they cannot be realized as infinite index subgroups of larger $\Theta$-Anosov subgroups.
	
	\begin{proposition}\label{prop:maximal_anosov}
		Let $\Gamma$ be a residually finite $\Theta$-Anosov subgroup of $G$. The following are equivalent:
		{
			\renewcommand{\theenumi}{(\roman{enumi})}
			\begin{enumerate}
				\item The $\Theta$-limit set $\Lambda\subset\F_\Theta$ of $\Gamma$ is locally maximally antipodal.
				\item The $\Theta$-limit set $\Lambda\subset\F_\Theta$ of $\Gamma$ is maximally antipodal.
				\item If $\Gamma'$ is a $\Theta$-Anosov subgroup of $G$ such that $[\Gamma: (\Gamma\cap \Gamma')] < \infty$, then  $\Gamma$ is commensurable with $\Gamma'$.
			\end{enumerate}
		}
	\end{proposition}
	
	We remark that the residual finiteness assumption is only needed to show (iii) implies (ii).
	
	\begin{proof}[Proof of Proposition \ref*{prop:maximal_anosov}]
		Clearly, (ii) implies (i). We show (i) implies (ii): Suppose that the $\Theta$-limit set $\Lambda\subset\F_\Theta$ of $\Gamma$ is locally maximally antipodal.
		Let $N$ be an open neighborhood of $\Lambda$ in $\F_\Theta$ such that for all $\tau\in N$, there exists $\lambda\in\Lambda$ such that $\tau$ is not antipodal to $\lambda$.
		If, on the contrary, $\Lambda$ is not maximally antipodal, there exists $\tau\in \F_\Theta$ that is antipodal to any point in $\Lambda.$ Let $\gamma\in\Gamma$ be an element with infinite order, and $\gamma_+$ denote the attractive fixed point of $\gamma$ in $\partial_\infty\Gamma.$ Since $\Lambda$ is the image of a $\Gamma$-equivariant, strongly dynamics-preserving boundary map from $\partial_\infty \Gamma$ to $\F_\Theta,$ the sequence $(\gamma^n \tau)$ accumulates at $\xi(\gamma_+) \in \Lambda \subset N.$
		However, as $\gamma^n \tau,$ for $n\in\mathbb{N},$ is antipodal to any point in $\Lambda,$ we can conclude that $\gamma^n\tau\not\in N$; this leads to a contradiction with the preceding sentence.
		
		Now, we show that (ii) implies (iii): Suppose that the $\Theta$-limit set $\Lambda$ of $\Gamma$ is maximally antipodal in $\F_\Theta$. If $\Gamma'$ is a $\Theta$-Anosov subgroup of $G$ such that  ${\rm H} = \Gamma\cap \Gamma'$ is a finite index subgroup of $\Gamma$, then ${\rm H}$ is also a $\Theta$-Anosov subgroup of $G$ with the same $\Theta$-limit set $\Lambda$. We proceed to show that ${\rm H}$ is a finite index subgroup of $\Gamma'$. As $\Theta$-Anosov subgroups are quasi-isometrically embedded in $G$, the inclusion ${\rm H}\hookrightarrow\Gamma'$ is a quasi-isometric embedding, which gives rise to a natural embedding $j: \partial_\infty{\rm H}\hookrightarrow \partial_\infty\Gamma'$. Moreover, the ${\rm H}$-equivariant map $\xi: \partial_\infty {\rm H} \to \Lambda$ is realized as the composition $\xi = \xi' \circ j$, where $\xi': \partial_\infty \Gamma' \to \F_\Theta$ is the strongly dynamics-preserving boundary map for $\Gamma'$. If $[\Gamma':{\rm H}] = \infty$, then $j: \partial_\infty{\rm H}\hookrightarrow \partial_\infty\Gamma'$ is not surjective, and therefore, $\Lambda = \xi' \circ j(\partial_\infty{\rm H})$ is a proper subset of the $\Theta$-limit set $\xi'(\partial_\infty\Gamma')$ of $\Gamma'$. Thus, since $\Lambda$ is maximally antipodal, $\xi'(\partial_\infty\Gamma')$ cannot be an antipodal subset of $\F_\Theta$. As $\Theta$-limit sets of $\Theta$-Anosov subgroups of $G$ are antipodal subsets of $\F_\Theta$, this leads to a contradiction.
		
		Finally, we show (iii) implies (ii): Suppose that the limit set $\Lambda$ in $\F_\Theta$ of $\Gamma$ is not maximally antipodal. Then, we can pick a pair of distinct points $\tau_\pm \in \F_\Theta \setminus\Lambda$ such that $\Lambda\cup \{\tau_\pm\}$ is antipodal.
		Let ${\rm H}$ be a cyclic $\Theta$-Anosov subgroup of $G$ with limit set $\{\tau_\pm\}$.
		Since $\Gamma$ and ${\rm H}$ are  residually finite, we can apply the Combination Theorem for Anosov subgroups \cite[Theorem 1.3]{DKL19} to obtain finite index subgroups $\Gamma_1 $ of $ \Gamma$ and ${\rm H}_1$ of ${\rm H}$ such that the subgroup $\Gamma'$ in $G$ generated by $\Gamma_1 $ and ${\rm H}_1$ is $\Theta$-Anosov and $\Gamma'$ is naturally isomorphic to the free product $\Gamma_1 \star {\rm H}_1$. Since $\Gamma_1 \subset \Gamma \cap \Gamma'$, it follows that $[\Gamma : (\Gamma\cap\Gamma')] \le [\Gamma : \Gamma_1] < \infty$. But $\Gamma_1$ is an infinite index subgroup of $\Gamma'$. Therefore, $\Gamma$ is not commensurable with $\Gamma'$.
	\end{proof}
	
	Following Kapovich-Leeb-Porti, the {\rm flag limit set} can be defined more generally for an arbitrary subgroup $G'<G$ \cite[Definition 4.25]{KLP17}; this agrees with the previous definition when $G'$ is a $\Theta$-Anosov subgroup. 
	The examples of Anosov subgroups we produce in \S\ref{sec:examples} arise implicitly as convex cocompact subgroups of rank $1$ subgroups $G'<G$. 
	For a uniform lattice $\Gamma$ in a rank 1 subgroup $G'$ inside $G$, the $\Theta$-limit set of $\Gamma$ equals the $\Theta$-limit set of $G'$.
	Therefore, our results in \S\ref{sec:sl2c maximally antipodal} and \S\ref{sec:complex hyperbolic space} are described in terms of flag limit sets of rank $1$ subgroups.
	
	\begin{remark}
		Guichard-Wienhard  \cite{GW18,GW22} have introduced the interesting notion of $\Theta$-positive representations of surface groups.
		Such representations are $\Theta$-Anosov, and triples in their $\Theta$-limit sets lie in components of pairwise transverse triples of flags which are called {\em $\Theta$-positive}.
		By \cite[Proposition 2.5(3)]{GLW21}, $\Theta$-positive triples are sent to another component of pairwise transverse triples under the inversion map $\iota$, see \Cref{def:inversion map}.
		When the inversion map does not leave a component of $C(\tau_\Theta) \cap C(\tau_\Theta^{\rm opp})$ invariant, arcs in that component (with endpoints at $\tau_\Theta,\tau_\Theta^{\rm opp}$) are locally maximally antipodal, again by the proof of \cite[Theorem A]{dey22}.
		It follows that limit sets of $\Theta$-positive representations are maximally antipodal.
		Hitchin representations and maximal representations are special cases of $\Theta$-positive representations. 
		These provide examples of maximally antipodal circles in $\Iso_n(\R^{2n},\omega)$. We record this observation in the following result.
	\end{remark}
	
	\begin{corollary}
		Let $\rho: \Gamma \to G$ be a $\Theta$-positive representation, where $\Gamma$ is a surface group. Then, the $\Theta$-limit set of $\rho(\Gamma)$ is a maximally antipodal subset of $\F_\Theta$. 
		
		In particular, if $\Gamma'$ is a $\Theta$-Anosov subgroup of $G$ such that $[\rho(\Gamma): (\rho(\Gamma)\cap \Gamma')] < \infty$, then  $\rho(\Gamma)$ is commensurable with $\Gamma'$.
	\end{corollary}
	
	\section{Anosov subgroups of the symplectic group} \label{sec:res_symplectic}
	
	Let $J \colon \R^{2n} \to \R^{2n}$ be the linear map defined by $Je_i = (-1)^i e_{2n-i+1}$ on the standard basis. Then $\omega(x,y)=x^TJy$ defines a symplectic form and the symplectic group is given by
	$$ \Sp(2n,\R) = \{ g \in {\rm GL}(2n,\mathbb{R}) :\ g^TJg=J\} .$$
	Observe that $J^2=-1,J=-J^T$, and $g \in \Sp(2n,\R)$ if and only if $-Jg^TJ=g^{-1}$. 
	
	\subsection{Restrictions on Anosov subgroups of the symplectic group}\label{sec:restrictionOnAnosov}
	
	The key restriction on Anosov subgroups comes from analyzing how the antiprincipal minors transform under inversion.
	
	\begin{definition}
		Define the \textit{antiprincipal $k \times k$ minor} of $g$ to be $p_k(g)$ where 
		$$ g e_{2n} \wedge g e_{2n-1} \wedge \cdots \wedge g e_{2n-k+1} \wedge e_{k+1} \wedge \cdots \wedge e_{2n} = p_k(g) e_1 \wedge e_2 \wedge \cdots \wedge e_{2n}. $$
	\end{definition}
	
	\begin{lemma}[Key Lemma]\label{lem:detCondition}
		For $g \in \Sp(2n,\R)$, $p_k(g^{-1}) = (-1)^k p_k(g)$.
	\end{lemma}
	
	We  use the notation $g[I,J]$ to denote the submatrix of $g$ formed by the rows $I$ and columns $J$; in particular, $p_k(g)= \det(g[\{1,\dots,k\},\{2n,\dots,2n-k+1\}])$.
	
	\begin{proof}
		We apply the definition:
		{
			\allowdisplaybreaks
			\begin{align*}
				p_k(g^{-1}) e_1 &\wedge e_2 \wedge \cdots \wedge e_{2n}\\
				& = g^{-1} e_{2n} \wedge g^{-1} e_{2n-1} \wedge \cdots \wedge g^{-1} e_{2n-k+1} \wedge e_{k+1} \wedge \cdots \wedge e_{2n} \\
				& = (-Jg^TJ e_{2n}) \wedge (-Jg^TJ e_{2n-1}) \wedge \cdots \wedge \\
				&\phantom{=}\hspace{1.5in} 
				(-Jg^TJ e_{2n-k+1}) \wedge e_{k+1} \wedge \cdots \wedge e_{2n} \\
				& = (g^TJ e_{2n}) \wedge (g^TJ e_{2n-1}) \wedge \cdots \wedge (g^TJ e_{2n-k+1}) \wedge Je_{k+1} \wedge \cdots \wedge Je_{2n} \\
				\intertext{To find the overall sign of applying $J$ we note that each of $e_{2n-k+1},\dots,e_{2n}$ appear twice. So it suffices to compute $(k+1)+\dots+(2n-k)= 2n^2-2nk+n-k\equiv n+k \mod 2$. We continue:}   
				& = (-1)^{n+k}(g^T e_{1}) \wedge (g^T e_{2}) \wedge \cdots \wedge (g^T e_{k}) \wedge e_{2n-k} \wedge \cdots \wedge e_{1} \\
				& = (-1)^{n+k} \det(g^T[\{2n,\dots,2n-k+1\},\{1,\dots,k\}])\\*
				&\phantom{=}\hspace{1.5in} e_{2n} \wedge e_{2n-1} \wedge \cdots \wedge e_{2n-k+1} \wedge e_{2n-k} \wedge \cdots \wedge e_{1} \\
				& = (-1)^{n+k} \det(g[\{1,\dots,k\},\{2n,\dots,2n-k+1\}])\\*
				&\phantom{=}\hspace{1.5in} e_{2n} \wedge e_{2n-1} \wedge \cdots \wedge e_{2n-k+1} \wedge e_{2n-k} \wedge \cdots \wedge e_{1} \\  
				& = (-1)^{k} \det(g[\{1,\dots,k\},\{2n,\dots,2n-k+1\}])\\*
				&\phantom{=}\hspace{1.5in} e_{1} \wedge e_{2} \wedge \cdots \wedge e_{k} \wedge e_{k+1} \wedge \cdots \wedge e_{2n} \\ 
				& = (-1)^{k}(g e_{2n}) \wedge (g e_{2n-1}) \wedge \cdots \wedge (g e_{2n-k+1}) \wedge e_{k+1} \wedge \cdots \wedge e_{2n} \\
				& = (-1)^{k} p_k(g) e_1 \wedge e_2 \wedge \cdots \wedge e_{2n}
			\end{align*}
		}
		which completes the proof.
	\end{proof}
	
	The flag manifolds defined in \S\ref{sec:FlagManifolds} have the following concrete description for $G = \Sp(2n,\R)$:
	Let us fix a subset $\Theta \subset \{1,\dots,n\}$. 
	An \textit{isotropic $\Theta$-flag} is a partial flag $V^{k_1} \subset \cdots \subset V^{k_\abs{\Theta}}$ in $(\R^{2n},\omega)$ with a component $V^i$ in each dimension $i \in \Theta$, such that each $V^i$ is \textit{isotropic}, i.e.\ the restriction of $\omega$ to $V^i$ is identically zero. 
	The space of isotropic $\Theta$-flags is naturally identified with the flag manifold $\F_\Theta$ associated to $\Sp(2n,\R)$, see \cite[p. 206]{Bou82}. 
	
	A pair of isotropic $\Theta$-flags $V,W$ is antipodal if, for all $i \in \Theta$, $V^i \oplus (W^i)^\perp = \R^{2n}$.
	Here we let $W^\perp$ denote the set of vectors $v$ satisfying $\omega(v,w)=0$ for all $w \in W$.
	In our setup, the standard flag $\tau_\Theta$, given by
	$$
	\tau_\Theta^i = {\rm Span}\{e_1,\dots,e_i\},\quad  i \in \Theta,
	$$ 
	is an isotropic $\Theta$-flag as is the standard opposite flag $\tau_\Theta^{\rm opp}$, given by
	$$
	(\tau_\Theta^{\rm opp})^i = {\rm Span}\{e_{2n},\dots,e_{2n-i+1}\}, \quad i \in \Theta.
	$$
	Note that $\tau_\Theta$ is antipodal to $\tau_\Theta^{\rm opp}$. 
	The stabilizer of $\tau_\Theta$ in $G = \Sp(2n,\R)$ is the parabolic subgroup $P_\Theta$ consisting of the intersection of $G$ with block upper triangular matrices preserving the partial flag corresponding to $\tau_\Theta$. 
	The unipotent radical $U_\Theta$ of $P_\Theta$ acts simply transitively on the flags antipodal to $\tau_\Theta$.\footnote{$U_\Theta$ is also called the {\em horocyclic subgroup} associated to $\Theta$.}
	
	We are interested in the space of pairwise antipodal triples of flags $(\tau_-,\tau,\tau_+)$ in $\F_\Theta$: Let us assume that $\tau_- = \tau_\Theta$ and $\tau_+ = \tau_\Theta^{\rm opp}$. 
	Since $\tau$ is antipodal to $\tau_\Theta$, we may write $\tau = u \tau_\Theta^{\rm opp}$ for a unique $u \in U_\Theta$. 
	Each $u \in U_\Theta$ is strictly upper triangular in the standard basis. 
	We want to understand when $u \tau_\Theta^{\rm opp}$ is antipodal to $\tau_\Theta^{\rm opp}$.
	
	\begin{lemma}\label{lem:transverseCondition}
		$\tau = u \tau_\Theta^{\rm opp}$ is antipodal to $\tau_\Theta^{\rm opp}$ if and only if for all $k \in \Theta$, $p_k(u) \ne 0$. 
	\end{lemma}
	
	\begin{proof} 
		We observe that $(u\tau_\Theta^{\rm opp})^k = \Span\{ue_{2n},\dots,ue_{2n-k+1}\}$ and $((\tau_\Theta^{\rm opp})^k)^\perp = \Span\{e_{2n},e_{2n-1},\dots,e_{k+1}\}$. 
		Therefore, $(u\tau_\Theta^{\rm opp})^k \oplus ((\tau_\Theta^{\rm opp})^k)^\perp = \R^{2n}$ if and only if \[ue_{2n}\wedge \cdots \wedge ue_{2n-k+1} \wedge e_{k+1} \wedge \cdots \wedge e_{2n} = p_k(u) e_1 \wedge e_2 \wedge \cdots \wedge e_{2n}\] is nonzero. 
	\end{proof}
	
	We are now in position to complete the proof of \Cref{thm:main}(i).
	
	\begin{lemma}\label{lem:symplectc_propI}
		Let $\Theta\subset \{1,\dots,n\}$.
		If $\Theta$ contains an odd integer, then $\mathcal{F}_\Theta$ has Property (I).
	\end{lemma}
	
	\begin{proof}
		Let $\Omega$ be any connected component of the intersection
		$C(\tau_{\Theta})\cap C(\tau_{\Theta}^{\rm opp})$. We would like to show that $\iota(\Omega) \cap \Omega = \emptyset$. Since $\iota$ preserves the intersection $C(\tau_{\Theta})\cap C(\tau_{\Theta}^{\rm opp})$ (see \Cref{lem:inversion_preserves}), it is enough to show that for some point $\tau\in \Omega$, $\iota(\tau)\not\in \Omega$. We argue by contradiction: 
		Given $\tau\in\Omega$, suppose, to the contrary, that there exists a continuous path $c:[-1,1] \to \Omega$ from $\iota(\tau) = c(-1)$ to $\tau = c(1)$.
		For $t\in[-1,1]$, let $u_t\in U_{\Theta}$ be the unique element such that $c(t) = u_t \tau_{\Theta}^{\rm opp}$. 
		Note that
		\[
		u_{-1} = u_1^{-1}.
		\]
		Let $k\in\Theta$ be an odd integer.
		By definition, $c(t)$ is antipodal to both $\tau_{\Theta}$ and $\tau_{\Theta}^{\rm opp}$, and so  by \Cref{lem:transverseCondition}, $p_k(u_t)\ne 0$ for all $t \in [-1,1]$.
		Since $t\mapsto u_t$ is continuous, the sign of $p_k(u_t)$ is constant for $t\in[-1,1]$. However,
		since $k$ is odd, by \Cref{lem:detCondition},
		\[
		p_k(u_{-1}) = -p_k(u_1),
		\]
		giving a contradiction.
	\end{proof}
	
	Together with \Cref{lem:symplectc_propI}, \Cref{thm:Anosovrestriction} implies the following:
	\begin{corollary}\label{cor:restriction_symplectic}
		Suppose that $\Theta\subset \{1,\dots,n\}$ contains an odd integer.
		If $\Gamma$ is a $\Theta$-Anosov subgroup of $\Sp(2n,\R)$, then $\Gamma$ is virtually isomorphic to a free group or a surface group.
	\end{corollary}
	
	\subsection{Examples from rank 1 subgroups}\label{sec:examples}
	
	In \S\ref{sec:restrictionOnAnosov}, we showed that if $\Theta \subset \{1,\dots,n\}$ contains an odd integer, then $\Theta$-Anosov subgroups of $\Sp(2n,\R)$ are virtually free or surface groups.
	In this section we show that this result is optimal.
	Let $\Theta_{\rm even}$ denote the set of all even integers in $\{1,\dots,n\}$, and let $\F_{\rm even}$ denote $\F_{\Theta_{\rm even}}$.
	We construct examples of $\Theta_{\rm even}$-Anosov subgroups of $\Sp(2n,\R)$, which are not virtually free or surface groups.
	
	\subsubsection{\texorpdfstring{$\Theta_{\rm even}$}{Theta even}-Anosov subgroups inside a copy of \texorpdfstring{$\SL(2,\C)$}{SL(2,C)}}\label{sec:examples from sl2c}
	
	When $n$ is even, any irreducible representation $\SL(2,\C) \to \SL(n,\C)$ preserves a symplectic form, so the image is contained in $\Sp(n,\C) \subset \Sp(2n,\R)$, up to conjugating the symplectic form to our standard one. 
	We let $\rho_n \colon \SL(2,\C) \to \Sp(2n,\R)$ denote this representation. 
	When $n$ is odd, we consider the representation $\rho_{n} \colon \SL(2,\C) \to \Sp(2n,\R)$ obtained by direct summing $\rho_{n-1}$ with a trivial $2$-dimensional representation.
	For concreteness, we fix the embedding $\Sp(2n,\R) \to \Sp(2n+2,\R)$ given by
	\begin{equation}\label{eqn:inclusion sp2n to sp2n+2}
		\begin{bmatrix} A & B \\ C & D \end{bmatrix} \mapsto \begin{bmatrix} A & 0 & B \\ 0 & I & 0 \\ C & 0 & D \end{bmatrix}
	\end{equation}
	where $I$ is the $2 \times 2$ identiy matrix.
	
	\begin{theorem}\label{thm:SL2Cexample}
		Let $n \ge 2$. If $\Gamma$ is a convex cocompact subgroup of $\SL(2,\C)$, then $\rho_n(\Gamma)$ is a $\Theta_{\rm even}$-Anosov subgroup of $\Sp(2n,\R)$.
	\end{theorem}
	
	Before discussing the proof, we briefly recall the Cartan decomposition and restricted roots for $\mfsp(2n,\R)$.
	The map 
	\[\theta \colon \mfsp(2n,\R) \to \mfsp(2n,\R), \hspace*{2pc} \theta(X) = -X^T\]
	is a Cartan involution of $\mfsp(2n,\R)$. 
	The fixed point set is $\mfk = \mathfrak{o}(2n,\R) \cap \mfsp(2n,\R)$ and its $-1$-eigenspace, denoted $\mfp$, is the intersection of $\mfsp(2n,\R)$ with the space of symmetric matrices.
	The Cartan decomposition is $\mfsp(2n,\R) = \mfk \oplus \mfp$.
	We let $\mfa$ denote the intersection of $\mfsp(2n,\R)$ with the space of diagonal matrices.
	Precisely, $\mfa = \{ \diag{\lambda_1,\dots,\lambda_n,-\lambda_n,\dots,-\lambda_1} :\ \lambda_i \in \R \}$.
	It is easy to see that $\mfa$ is a maximal abelian subspace of $\mfp$, since it contains a diagonal element with distinct entries.
	The restricted roots $\Sigma \subset \mfa^\ast$ are the nonzero weights of the adjoint representation of $\mfa$ on $\mfg$. 
	In this case, $\Sigma \cup \{0\}= \{ \pm \lambda_i \pm \lambda_j\}_{ij}$.
	We take the positive Weyl chamber to be the subset of $\mfa$ with strictly decreasing entries. 
	The corresponding set of simple roots is then $\Delta=\{\lambda_1 -\lambda_2,\lambda_2-\lambda_3,\dots,\lambda_{n-1}-\lambda_n,2\lambda_n\}$.
	We label the simple roots by $\alpha_i = \lambda_i-\lambda_{i-1}$ for $i <n$ and $\alpha_n=2\lambda_n$. 
	Then the subset $\Theta_{\rm even}$ is $\{\alpha_2,\dots, \alpha_{2\lfloor \frac{n}{2}\rfloor}\}$, i.e.\ the subset of $\Delta$ consisting of simple roots labelled by even integers.
	Observe that the embedding given by (\ref{eqn:inclusion sp2n to sp2n+2}) preserves the choices above and maps the positive Weyl chamber to the nonnegative Weyl chamber.
	
	\begin{proof}[Proof of Theorem \ref*{thm:SL2Cexample}]
		We first consider the case when $n$ is even:
		The representation $\rho_n$ induces a representation of $\mfsl(2,\C)$.
		Up to conjugation, this Lie algebra representation maps 
		\[ H' = \begin{bmatrix} 1 & 0 \\ 0 & -1 \end{bmatrix} \mapsto H= \diag{(n-1)I,(n-3)I, \dots,(1-n)I} ,\]
		where $H$ is expressed as a block diagonal matrix of $2 \times 2$ blocks ($I$ denotes the $2 \times 2$ identity matrix); this is easy to see directly but also follows from \Cref{prop:standard sl2C rep} below.
		Since the simple roots in $\Theta_{\rm even}$ are positive on $H$, we may apply \cite[Proposition 4.7]{GW12}, which says that for any convex cocompact subgroup $\Gamma$ of $\SL(2,\C)$, $\rho_n(\Gamma)$ is $\Theta_{\rm even}$-Anosov in $\Sp(2n,\R)$. 
		(We remark that our subset $\Theta$ would be called $\Delta \setminus \Theta$ in the conventions of \cite{GW12}.)
		
		For $n$ odd,
		we start with the matrix $H$ for $\rho_{n-1}$.
		Applying the inclusion (\ref{eqn:inclusion sp2n to sp2n+2}) to $H$ yields a matrix on which each simple root in $\Theta_{\rm even}$ is positive, and \cite[Proposition 4.7]{GW12} applies once again.
	\end{proof}
	
	\subsubsection{The \texorpdfstring{$\Theta_{\rm even}$}{Theta even}-limit set of \texorpdfstring{$\rho_n(\SL(2,\C))$}{SL(2,C)} is maximally antipodal}\label{sec:sl2c maximally antipodal}
	
	In \Cref{thm:SL2Cexample}, we obtained antipodal subsets of $\F_{\rm even}$ homeomorphic to $S^2$ as a limit set of $\rho_n(\Gamma)$, where $\Gamma<\SL(2,\C)$ is any uniform lattice.
	In this section we show these subsets are maximally antipodal, see \Cref{cor:sl2c limit set is MA in Iso2 and F_even}.
	In fact we will show something slightly stronger.
	Via the natural inclusion $\Sp(2n,\R) \subset \SL(2n,\R)$, a $\Theta_{\rm even}$-Anosov subgroup of $\Sp(2n,\R)$ may be viewed as an Anosov subgroup of $\SL(2n,\R)$ with respect to the subset of simple roots labelled by even integers $2 \le k \le 2n-2$, see \cite[Proposition 4.4]{GW12}.
	In particular, such a subgroup is $\{2,2n-2\}$-Anosov in $\SL(2n,\R)$ and has a flag limit set in the flag manifold $\mathcal{F}_{2,2n-2}$ of $\SL(2n,\R)$ consisting of pairs $(x,y)$ where $x$ is a $2$-plane in $\R^{2n}$ and $y$ is a codimension $2$-plane containing $x$.
	We show in \Cref{thm:sl2c limit set is maximally antipodal} that the limit set of $\rho_n(\SL(2,\C))$ in $\mathcal{F}_{2,2n-2}$ is maximally antipodal.
	
	In the analysis below, we view all matrices as $2\times2$ block matrices. 
	We fix the following basis of $2\times 2$ matrices:
	\begin{equation}\label{eqn:M2basis}
		I= \begin{bmatrix} 1 & 0 \\ 0 & 1 \end{bmatrix} \hspace{1pc} 
		T= \begin{bmatrix} 1 & 0 \\ 0 & -1 \end{bmatrix} \hspace{1pc} 
		R=\begin{bmatrix} 0 & -1 \\ 1 & 0 \end{bmatrix} \hspace{1pc} 
		P= \begin{bmatrix} 0 & 1 \\ 1 & 0 \end{bmatrix}  
	\end{equation}
	We note that the nonzero matrices in the span of $I$ and $R$ have positive determinant while the nonzero matrices in the span of $T$ and $P$ are traceless symmetric with negative determinant. 
	
	\begin{theorem}\label{thm:sl2c limit set is maximally antipodal}
		Let $\Lambda$ be the $\{2,2n-2\}$-limit set in $\mathcal{F}_{2,2n-2}$ of $\rho_n(\SL(2,\C))\subset \SL(2n,\R)$.
		Then $\Lambda$ is maximally antipodal in $\mathcal{F}_{2,2n-2}$.
	\end{theorem}
	
	\begin{proof} 
		Let $\tau_+ \in \mathcal{F}_{2,2n-2}$ be the partial flag defined by $(\tau_+)^2=\Span\{e_1,e_2\}$ and $(\tau_+)^{2n-2}=\Span\{e_1,e_2,\dots,e_{2n-2}\}$.
		The horocyclic subgroup of $\SL(2n,\R)$ corresponding to $\tau_+$ is given by
		\[U = \left\{\begin{bmatrix} I & A & B \\ 0 & I_{2n-4} & C \\ 0 & 0 & I \end{bmatrix} \right\} \subset \SL(2n,\R) \]
		and acts simply transitively on $C(\tau_+) \subset \mathcal{F}_{2,2n-2}$.
		Let $\tau_- \in \mathcal{F}_{2,2n-2}$ be the partial flag with $(\tau_-)^2=\Span\{e_{2n},e_{2n-1}\}$ and $(\tau_-)^{2n-2}=\Span\{e_{2n},e_{2n-1},\dots,e_{3}\}$, and note that $\tau_-$ is transverse to $\tau_+$.
		We will show that for every $g\in U$, $g\tau_{-}$ is not antipodal to some point in $\Lambda$, up to an arbitrarily small perturbation of $g$.
		This suffices as the condition of being non-antipodal is closed: 
		to see this, consider 
		$$ \mathcal{E} = \{(y,f) \in \Lambda \times \mathcal{F} :\ f \in E(y) \} $$
		where $E(y) = \mathcal{F} \setminus C(y)$ is the set of flags non-antipodal to $y$.
		Note that each $E(y)$ is compact, so the fibers of $\mathcal{E} \to \Lambda$ are compact and the base is compact. 
		It follows that $\mathcal{E}$ is compact; indeed, since $\mathcal{F}$ is metrizable, it suffices to show that $\mathcal{E}$ is sequentially compact.
		Given a sequence $(y_n,f_n)$ in $\mathcal{E}$, we can assume that  $(y_n)$ converges to $y$ in $\Lambda$ up to passing to a subsequence.
		By passing to a further subsequence, we may assume that $f_n$ converges to $f$ in $\mathcal{F}$.
		Since each $f_n \in E(y_n)$, the limit $f$ is contained in $E(y)$, so $\mathcal{E}$ is sequentially compact.
		Therefore the image of $\mathcal{E} \to \mathcal{F}$ is compact, and this image is exactly the set of flags non-antipodal to some point of $\Lambda$. 
		
		First consider the case where $n$ is even. 
		We let $H,X,Y$ denote the images of $H',X',Y'$ under the representation $\rho_n \colon \mfsl(2,\C) \to \mfsp(2n,\R)$.
		By \Cref{prop:standard sl2C rep} below we may assume that
		\begin{equation}\label{eqn:explicitEvenRep}
			\begin{aligned}
				H =~& \diag{(n-1)I, (n-3)I, \dots , (1-n)I}\\
				X =~& \superdiag{1}{c_1 I,\dots, c_{n/2-1}I,c_{n/2}T,-c_{n/2+1}I,\dots,-c_{n}I}\\
				Y =~& \superdiag{1}{c_1 R,\dots,c_{n/2-1}R,c_{n/2}P,c_{n/2+1}R,\dots,c_{n}R}
			\end{aligned} 
		\end{equation}
		where $c_k = \sqrt{kn-k^2}$. 
		We have 
		\[ 
		\Lambda \setminus \{\tau_+\}  = \{\exp(\alpha X + \beta Y)\tau_- :\  \alpha, \beta \in \R \} .
		\] 
		Consider $\exp(\alpha X + \beta Y)$ as a block $2 \times 2 $ matrix. 
		The entries are polynomials in $\alpha, \beta$. 
		The degree of a $2 \times 2$ block of $\exp(\alpha X + \beta Y)$ is $k$ when that block is on the $k$th diagonal.
		In particular, the highest degree terms are the top-right block of $\exp(\alpha X + \beta Y)$, and the degree is $n-1$, hence odd.
		
		The top right block of $\exp(\alpha X + \beta Y)$ is traceless symmetric, as we now explain. Looking at the explicit representation in \eqref{eqn:explicitEvenRep}, we observe that 
		\[ \alpha X + \beta Y = \superdiag{1}{C_1,\dots,C_{n-1}} \]
		where $C_k=\alpha A_k+\beta B_k$.
		In particular, $C_{n/2}$ is in the span of $T$ and $P$  and for $k \neq n/2$, $C_k$ is in the span of $I$ and $R$.
		The top right block of $\exp(\alpha X + \beta Y)$ is then the product 
		\[ \exp(\alpha X + \beta Y)_{1n} = \frac1{(n-1)!} C_1C_2 \cdots C_{n/2-1} C_{n/2} C_{n/2+1} \cdots C_{n-1} .\] 
		Multiplying any matrix in the span of $T$ and $P$ on the left or right by a matrix in the span of $I$ and $R$ remains in the span of $T$ and $P$. 
		Thus $\exp(\alpha X+\beta Y)_{1n}$ is in the span of $T$ and $P$ and therefore is traceless symmetric. 
		
		Any other flag $g^{-1} \tau_- \in C(\tau_+) \subset \mathcal{F}_{2,2n-2}$ is transverse to $\Lambda$ if and only if the second antiprincipal minor $p_2(g \exp(\alpha X + \beta Y))$ is nonvanishing for all $\alpha,\beta \in \R$ by \Cref{lem:transverseCondition}. 
		We will show that for any $g \in U$ there exists $g'$ arbitrarily near $g$ and $\alpha,\beta$ such that $p_2(g' \exp(\alpha X + \beta Y))=0$. 
		
		The top-right block of $g \exp(\alpha X + \beta Y)$ is $Z_{1n}=\sum_j g_{1j} \exp(\alpha X + \beta Y)_{jn}$.
		As a polynomial in $\alpha, \beta$, the highest degree term is $g_{11}\exp(\alpha X + \beta Y)_{1n}$ which has degree $n-1$.
		Since $g_{11}$ is the identity matrix this term is traceless symmetric. 
		
		We now consider the components of $Z_{1n}$ in the basis $I,R,T,P$, see \eqref{eqn:M2basis}.
		Observe that the coefficients of $T$ and $P$ have a higher degree than the coefficients of $I,R$, so for sufficiently large $\alpha,\beta$ the determinant of $Z_{1n}$ is negative. 
		
		The coefficients of $T$ and $P$ have a common real root, up to an arbitrarily small perturbation of $g$.
		Indeed, let $f_T(\alpha,\beta)$ (resp.\ $f_P(\alpha,\beta)$) be the coefficient of $T$ (resp.\ $P$) in $Z_{1n}$.
		$f_T$ and $f_P$ are real polynomials in the variables $\alpha,\beta$. 
		Let $\overline{f_T}$ (resp.\ $\overline{f_P}$) denote the sum of highest degree terms of $f_T$ (resp.\ $f_P$).
		In fact, $\overline{f_T}$ (resp.\ $\overline{f_P}$) is independent of $g$, and equals the coefficient of $T$ (resp.\ $P$) in $\exp(\alpha X + \beta Y)_{1n}$. 
		It is convenient to consider the homogenizations $\widehat{f_T}(\alpha,\beta,\gamma)$ and $\widehat{f_P}(\alpha,\beta,\gamma)$, i.e.\ homogeneous polynomials such that $\widehat{f_T}(\alpha,\beta,1)=f_T(\alpha,\beta)$ and $\widehat{f_P}(\alpha,\beta,1)=f_P(\alpha,\beta)$.
		Note that we can modify the constant terms of $f_T$, $f_P$ by modifying $g_{1n}$. 
		Therefore, up to an arbitrarily small perturbation of $g$, we can assume that the zero sets of $\widehat{f_T}$ and $\widehat{f_P}$ in $\mathbb{CP}^2$ have no common algebraic components. 
		Then by Bezout's theorem, there are exactly $(n-1)^2$ common zeros of $\widehat{f_T}$ and $\widehat{f_P}$ in $\mathbb{CP}^2$, counted with multiplicity. 
		Since $n$ is even, there are an odd number of zeros; since complex conjugation permutes the roots preserving multiplicity, there exists a real root. 
		It then suffices to rule out the possibility of a real root ``at infinity,'' i.e.\ when $\gamma=0$. 
		A real root at infinity for $\widehat{f_T}$ and $\widehat{f_P}$ corresponds to a nonzero root of $\overline{f_T}$ and $\overline{f_P}$. 
		But a common real root of $\overline{f_T}$ and $\overline{f_P}$ is a pair of real numbers $(\alpha,\beta)$ such that $\exp(\alpha X + \beta Y)_{1n}=0$. 
		This can only occur when $(\alpha,\beta)=(0,0)$, e.g.\ by transversality of the limit set.
		
		At a common real root of $f_T$ and $f_P$, the block $Z_{1n}$ is in the span of $R$ and $I$ and so has determinant $\ge 0$. 
		If it is zero, we are done; otherwise, we apply the intermediate value theorem, to see that a zero exists, and then we are done. 
		
		We now consider $\rho_{n}$ where $n$ is odd.
		The homomorphism $\rho_{n}$ is induced by $\rho_{n-1}$ and the inclusion (\ref{eqn:inclusion sp2n to sp2n+2}).
		Then $H=\rho_{n}(H')$ is still diagonal with decreasing entries, but $X$ and $Y$ no longer have all entries on the superdiagonal.
		However, the top-right block for $\exp(\alpha X + \beta Y)$ under $\rho_{n}$ agrees with the top-right block for $\rho_{n-1}$.
		To see this, consider a simultaneous permutation of the rows and columns swapping the central $2 \times 2$ block with the bottom-right $2 \times 2$ block:
		\[
		\begin{bmatrix} A & 0 & B \\ 0 & I & 0 \\ C & 0 & D \end{bmatrix} \mapsto \begin{bmatrix} A & B & 0 \\ C & D & 0 \\ 0 & 0 & I \end{bmatrix}.
		\]
		Then as before, the top right block of $\exp(\alpha X + \beta Y)$  is traceless symmetric, has strictly larger degree than any other block, and this degree is odd.    
		From this point we may apply the same proof as the even case.   
	\end{proof}
	
	\begin{corollary}\label{cor:sl2c limit set is MA in Iso2 and F_even}
		The $\{2\}$-limit set $\Lambda_{\Iso_2}$ of $\rho_n(\SL(2,\C))<\Sp(2n,\R)$ is maximally antipodal in $\Iso_2(\R^{2n},\omega)$.
		Moreover, the $\Theta_{\rm even}$-limit set $\Lambda_{\rm even} \subset \mathcal{F}_{\rm even}$ of $\rho_n(\SL(2,\C)) < \Sp(2n,\R)$ is also maximally antipodal.
	\end{corollary}
	
	\begin{proof}
		The isotropic flag manifold $\Iso_2(\R^{2n},\omega)$ for $\Sp(2n,\R)$ naturally embeds into the flag manifold $\mathcal{F}_{2,2n-2}$ of $\SL(2n,\R)$ via the map $V\mapsto(V \subset V^\perp)$. 
		By \Cref{thm:sl2c limit set is maximally antipodal}, $\Lambda$ is maximally antipodal in $\mathcal{F}_{2,2n-2}$. 
		Then in particular, every isotropic flag in $\Iso_2(\R^{2n},\omega) \subset \mathcal{F}_{2,2n-2}$ is non-transverse to some point of $\Lambda$. 
		So $\Lambda_{\Iso_2}$ of is maximally antipodal in $\Iso_2(\R^{2n},\omega)$.
		
		Now consider an arbitrary element $\tau$ of $\mathcal{F}_{\rm even}$.
		By the previous paragraph, the $2$-dimensional part of $\tau$ is non-transverse to some point of $\Lambda_{\Iso_2}$. 
		So $\Lambda_{\rm even}$ is maximally antipodal in  $\mathcal{F}_{\rm even}$.
	\end{proof}
	
	\begin{proposition}\label{prop:standard sl2C rep}
		Let $n$ be even. 
		Let $\rho \colon \mfsl(2,\C) \to \mfsp(2n,\R)$ be a representation such that the eigenvalues  of 
		$$ \rho \left( \begin{bmatrix} 1 & 0 \\ 0 & -1 \end{bmatrix} \right) $$
		are $(n-1,n-1,n-3,n-3,\dots, n-2k+1,\dots,1-n,1-n)$ (with multiplicity).
		Then, up to conjugation by an element of $\Sp(2n,\R)$, $\rho$ intertwines the Cartan involutions of $\mfsl(2,\C)$ and $\mfsp(2n,\R)$ and moreover maps
		\begin{equation}
			\begin{aligned}
				H' = \begin{bmatrix} 1 & 0 \\ 0 & -1 \end{bmatrix} \mapsto H =~& \diag{(n-1)I, (n-3)I, \dots , (1-n)I}\\
				X' = \begin{bmatrix} 0 & 1 \\ 0 & 0 \end{bmatrix} \mapsto X =~& \superdiag{1}{c_1 I,\dots, c_{n/2-1}I,c_{n/2}T,-c_{n/2+1}I,\dots,-c_{n}I} \\
				Y' = \begin{bmatrix} 0 & i \\ 0 & 0 \end{bmatrix} \mapsto Y =~& \superdiag{1}{c_1 R,\dots,c_{n/2-1}R,c_{n/2}P,c_{n/2+1}R,\dots,c_{n}R}
			\end{aligned} 
		\end{equation}
	\end{proposition}
	
	It will be convenient to set the following notation which appears in the proof.
	\begin{definition}\label{def:bar of 2x2 matrix} 
		We let $\overline{A}$ denote the {\em adjugate} of a $2\times 2$ matrix $A$: 
		\[ A= \begin{bmatrix}a &  b\\ c & d \end{bmatrix} \mapsto \begin{bmatrix} d & -b \\ -c & a \end{bmatrix} = \overline{A} \]
	\end{definition}
	We easily observe that $\overline{I} = I$ and $\overline{T} = -T, \overline{R} = -R, \overline{P} = -P$.
	
	\begin{proof}[Proof of Proposition \ref*{prop:standard sl2C rep}]
		By the Karpelevic-Mostow Theorem \cite{Kar53,Mos55}, we may assume that $\rho$ intertwines the Cartan involutions $\theta'(Z)=-\overline{Z}^T$ of $\mfsl(2,\C)$ and $\theta(Z)=-Z^T$ of $\mfsp(2n,\R)$.
		It then follows that it takes $H'$ to $\mfp$.
		Then up to conjugation in $K$ we may assume that $\rho(H')$ is diagonal, since $K$ acts transitively on the maximal abelian subspaces of $\mfp$. 
		Moreover, the Weyl group acts transitively on chambers, so we may further conjugate so that $\rho(H')=H$.
		
		We now show that we can conjugate $\rho$ so that $X'$ maps to $X$, except possibly at the middle block.
		The matrices $Z$ satisfying $[H,Z]=2Z$ are exactly the super-diagonal matrices given by $Z = \superdiag{1}{A_1,\dots,A_{n-1}}$.
		Such a $Z$ is in $\mfsp(2n,\R)$ if and only if $A_{n-k}=-\overline{A_k}$.
		The Levi subgroup is $L = Z_{\Sp(2n,\R)}(H)=\{\diag{g_1,\dots,g_n}: g_{n-k+1}^{-1} = \overline{g_k} \}$. 
		It acts on superdiagonal matrices by 
		\[
		\superdiag{1}{\dots, A_k ,\dots} \mapsto \superdiag{1}{\dots,g_k A_k g_{k+1}^{-1},\dots}.
		\]
		We will only conjugate $\rho$ by the subgroup $L \cap K$ where each $g_k$ is orthogonal, so that we preserve the Cartan involution.
		We know that $[\rho(X'),\rho(X'^T)]=\rho([X',X'^T])=\rho(H')=H$; it follows that each $A_k$ appearing in $\rho(X')$ is invertible. 
		Since $\rho$ intertwines the Cartan involution, it follows that each $A_k$ is an orthogonal matrix times $c_k = \sqrt{kn-k^2}$.
		Therefore, by setting $g_1=I$, the equations $g_kA_k=c_kg_{k+1}$ determine an element of $L$ taking $\rho(X')$ to $\superdiag{1}{c_1 I,\dots, c_{n/2-1}I,A,-c_{n/2+1}I,\dots,-c_{n}I}$ for some $A$ satisfying $A=-\overline{A}$. 
		Now the subgroup of $L$ given by $\{\diag{g,\dots,\overline{g}^{-1}}\}$ preserves all the blocks of $X$ except possibly the middle block $A$, and it takes $A$ to $gA\overline{g} = \det(g) gAg^{-1}$. 
		
		In order to control the middle block $A$, we must consider the image of $Y'$.
		As before, $\rho(Y')$ is superdiagonal with blocks $B_k$. 
		The fact that $[X',Y']=0$ implies that $c_{k+1}B_k=c_kB_{k+1}$ for $1 \le k \le n/2-1$. 
		We need to see that $\det(A_k)$ and $\det(B_k)$ have the same sign for all $k$.
		Observe that the top right block of $\exp(\alpha\rho_n(X')+\beta\rho_n(Y'))$ is the product 
		\[ \exp(\alpha \rho(X') + \beta\rho(Y'))_{1n} = \frac1{(n-1)!} C_1C_2 \cdots C_{n/2-1} C_{n/2} (-\overline{C_{n/2}}) \cdots (-\overline{C_{1}}) \]
		where $C_k = \alpha A_k+\beta B_k$.
		By transversality, this block has nonzero determinant when $\alpha,\beta$ are not both zero.
		It follows that $\{\alpha A_k + \beta B_k: \alpha,\beta \in \R \}$ is a definite subspace of $(M_2(\R),\det)$ for all $k$, so indeed $\det(A_k)$ and $\det(B_k)$ have the same sign.
		The middle block is traceless, so must have negative determinant.
		Now the middle block is, up to rescaling by $c_{n/2}$, orthogonal with negative determinant. 
		Hence we can conjugate so that $A_{n/2}$ becomes $c_{n/2}T$.
		
		At this point, we have shown that $\rho$ can be conjugated to intertwine the Cartan involutions, take $H'$ to $H$, and take $X'$ to $X$.
		We know that $\rho(Y')$ is a superdiagonal block matrix of the form $\superdiag{1}{c_1B_0,c_2B_0,\dots,c_{n/2}B_0',\dots}$ with $B_0$ and $B_0'$ orthogonal, $\det(B_0)=1$ and $\det(B_0')=-1$.
		Since the bracket $[X',Y']=0$ we have $B_0'=B_0T$.
		We now consider the inner product $B_\theta(U,V)=-B(\theta U,V)$ on $\mfsp(2n,\R)$. 
		Since $\rho$ is injective and intertwines the Cartan involutions the pullback $\rho^\ast B_\theta$ agrees with $B_{\theta'}$ up to a constant scalar.
		Therefore $B_\theta(\rho(X'),\rho(Y'))=0$.
		This implies that $B_0$ is traceless; since moreover $B_0$ is orthogonal with determinant $1$, it must be $\pm R$.
		Then $B_0'=B_0T = \pm P$.
		
		We need to show that we can further conjugate $\rho(Y')$ to $Y$ while keeping all of the data above preserved. 
		For this we are only able to conjugate using the group $L \cap K \cap Z_{\Sp(2n,\R)}(X)$, 
		which contains only the block diagonal elements where $g \in \{\pm I,\pm P\}$.
		Up to conjugating by the block diagonal matrix of $P's$, we have that $B_0=R$.
		Then $B_0'=RT=P$. 
	\end{proof}
	
	\subsection{Maximal antipodality of the limit set of \texorpdfstring{$\SU(n-1,1)$}{SU(n-1,1)}}\label{sec:complex hyperbolic space}
	
	We let $\SU(n-1,1)$ denote the subgroup of $\SL(n,\C)$ preserving the indefinite hermitian form $h(u,v)=\overline{u}^TQv$ where $Q(e_1)=e_n$, $Q(e_n)=e_1$ and $Q(e_k)=e_k$ for $1<k<n$.
	We have 
	\[ \mathfrak{su}(n-1,1) = \left\{ \begin{bmatrix} a-\frac12 \Tr(X) & -\overline{v}^T & ib \\ u & X & v \\ ic & -\overline{u}^T & -a-\frac12 \Tr(X) \end{bmatrix}
	:\
	\begin{array}{l}
		a,b,c \in \R,\\
		u,v \in \C^{n-2},\\
		X\in \mathfrak{u}(n-2)
	\end{array}
	\right\} .\]
	Note that $\Tr(X) \in i\R$.
	The Cartan involution $\theta(Z)=-\overline{Z}^T$ preserves $\mathfrak{su}(n-1,1)$ and induces the Cartan decomposition $\mathfrak{su}(n-1,1) = \mfk \oplus \mfp$.
	Here $\mfk$ is the intersection $\mfk = \mathfrak{su}(n-1,1) \cap \mathfrak{u}(n)$, i.e.\ the intersection of $\mathfrak{su}(n-1,1)$ with skew-Hermitian matrices and $\mfp$ is the intersection of $\mathfrak{su}(n-1,1)$ with Hermitian matrices. 
	The intersection of $\mfp$ with real diagonal matrices is the maximal abelian subspace $\mfa \subset \mfp$ given by
	\[
	\mfa = \left\{ \begin{bmatrix} x & 0 & \cdots & 0 & 0 \\ 0 & 0 & \cdots & 0 & 0 \\ \vdots & \vdots & \ddots & \vdots & \vdots \\ 0 & 0 & \cdots & 0 & 0 \\ 0 & 0 & \cdots & 0 & -x \end{bmatrix}  
	:\
	x \in \R \right\} .
	\] 
	In the induced restricted root space decomposition 
	\[ \mathfrak{su}(n-1,1) = \mfg_{-2\alpha} \oplus \mfg_{-\alpha} \oplus \mfa \oplus \mathfrak{m} \oplus \mfg_\alpha \oplus \mfg_{2\alpha} \]
	we have 
	\[ \mfg_\alpha = \left\{ \begin{bmatrix} 0 & -\overline{z_1} & \cdots & -\overline{z_n} & 0 \\ 0 & 0 & \cdots & 0 & z_1 \\ \vdots & \vdots & \ddots & \vdots & \vdots \\ 0 & 0 & \cdots & 0 & z_n \\ 0 & 0 & \cdots & 0 & 0 \end{bmatrix} 
	:\
	z_1, \dots, z_n \in \C \right\} \]
	and 
	\[ \mfg_{2\alpha} = \left\{ \begin{bmatrix} 0 & 0 & \cdots & 0 & iy \\ 0 & 0 & \cdots & 0 & 0 \\ \vdots & \vdots & \ddots & \vdots & \vdots \\ 0 & 0 & \cdots & 0 & 0 \\ 0 & 0 & \cdots & 0 & 0 \end{bmatrix} 
	:\
	y \in \R \right\} .\]
	
	We identify $\C^n$ with $\R^{2n}$ via $(x_1+iy_1,\dots,x_n+iy_n)\mapsto(x_1,y_1,\dots,x_n,y_n)$.
	The imaginary part of $h$ is a real-valued symplectic form $\omega_h$ preserved by $\SU(n-1,1)$.
	Under the identification of $\C^n$ with $\R^{2n}$ it may be written as $\omega_h(x,y)=x^T J_h y$ for the matrix 
	\[ J_h = \begin{bmatrix} 0 & 0 & \cdots & 0 & R \\ 0 & R & \cdots & 0 & 0 \\ \vdots & \vdots & \ddots & \vdots & \vdots \\ 0 & 0 & \cdots & R & 0 \\ R & 0 & \cdots & 0 & 0 \end{bmatrix} ,\]
	which has the $2 \times 2$ block $R$ in the top-right, lower-left, and otherwise has $R$ down the diagonal blocks. 
	
	Under this identification, the horocyclic subalgebra $\mfu'=\mfg_\alpha\oplus \mfg_{2\alpha}$ is given by
	\[ \mfu' = \left\{ \begin{bmatrix} 0 & -\alpha^T \otimes I + \beta \otimes R & \gamma R \\ 0 & 0 & \alpha \otimes I + \beta \otimes R \\ 0 & 0 & 0 \end{bmatrix} 
	:\
	\begin{array}{l}
		\gamma \in \R,\\
		\alpha,\beta \in \R^{n-2}
	\end{array}
	\right\} \]
	where e.g.\ $\alpha \otimes I$ denotes an $n-2 \times 2$ block matrix according to the Kronecker product.
	It exponentiates to the group
	\begin{align*}
		U' &= \exp(\mfu') \\
		&=
		\left\{
		\begin{bmatrix} I & -\alpha^T \otimes I + \beta \otimes R & -\frac12(\abs{\alpha}^2+\abs{\beta}^2)I + \gamma R \\ 0 & I & \alpha \otimes I + \beta \otimes R \\ 0 & 0 & I \end{bmatrix} 
		:\
		\begin{array}{l}
			\gamma \in \R,\\
			\alpha,\beta \in \R^{n-2}
		\end{array}
		\right\}.
	\end{align*}
	where $\abs{\alpha}^2$ denotes the standard norm squared of $\alpha$. 
	Note that $U'$ acts simply transitively on $\Lambda \setminus \{\tau_+\}$.
	
	We also consider the horocyclic subgroup $U$ of $\Sp(\omega_h)$ acting simply transitively on  $C(\tau_+)\subset \Iso_2$.
	This is given by
	\[
	U = 
	\scalemath{0.95}{
		\left\{
		\begin{bmatrix} I & -\overline{F(u,v,w,z)}^T & q(u,v,w,z)I +bR+cT+dP \\ 0 & I & F(u,v,w,z) \\ 0 & 0 & I \end{bmatrix}
		:\
		\begin{array}{l}
			b,c,d \in \R,\\
			u,v,w,z \in \R^{n-2}
		\end{array}
		\right\}
	}
	\]
	where, for $u,v,w,z \in \R^{n-2}$, we set
	\begin{align*}
		q(u,v,w,z) & = \frac12\left(-\abs{u}^2-\abs{v}^2+\abs{w}^2+\abs{z}^2\right), \\
		F(u,v,w,z) & = u\otimes I+v\otimes R+w\otimes T+z\otimes P, \\
		-\overline{F(u,v,w,z)}^T & = -u^T\otimes I+v^T\otimes R+w^T\otimes T+z^T\otimes P.
	\end{align*}
	To see that $U$ can be parameterized in this way, observe that each element can be factored into a product of a matrix of the form
	\[ \exp \left( \begin{bmatrix} 0 & -\overline{F(u,v,w,z)}^T & 0 \\ 0 & 0 & F(u,v,w,z) \\ 0 & 0 & 0 \end{bmatrix} \right)  \]
	and a matrix of the form
	\[\exp \left( \begin{bmatrix} 0 & 0 & bR+cT+dP \\ 0 & 0 & 0 \\ 0 & 0 & 0 \end{bmatrix} \right) .\]
	
	\begin{theorem}\label{thm:MaximalAntipodalSU}
		The limit set $\Lambda \cong S^{2n-3} \subset \Iso_2(\R^{2n},\omega_h)$ associated to $\SU(n-1,1)\subset \Sp(\omega_h)$ is maximally antipodal.
	\end{theorem}
	
	Before presenting the proof, we first explain how the symplectic form $\omega_h$ is related to the symplectic form $\omega$ we discussed above.
	It is convenient to fix a transformation $f \colon \R^{2n} \to \R^{2n}$ which relates the two symplectic forms.
	When $n$ is even, we may use 
	\[ f = \frac1{\sqrt{2}} \begin{bmatrix} 
		\sqrt{2}I &  & &  &  &  &  & 0 \\ 
		& I &  &  &  &  & I &  \\
		&  & \ddots &  &  & \scalebox{-1}[1]{$\ddots$} &  & \\
		&  &  & I & I &  &  &  \\
		&  &  & T & -T &  &  &  \\
		&  & \scalebox{-1}[1]{$\ddots$} &  & & \ddots &  &  \\
		& T &  &  &  &  & -T &  \\
		0 &  &  &  &  &  &  & \sqrt{2}I \end{bmatrix} \]
	and when $n$ is odd, we modify $f$ by inserting a middle row and column with $\sqrt{2}I$ in the center block and zeros elsewhere. 
	It is easy to check that $f J f^T=J_h$, and it then follows directly that $g \in \Sp(\omega_h)$ if and only if $f^Tgf^{-T} \in \Sp(\omega)$ if and only if $f^{-1}gf\in \Sp(\omega)$.
	Moreover, a subspace $V$ is $\omega_h$-isotropic if and only if $f^TV$ is $\omega$-isotropic. 
	In particular, the standard isotropic $2$-flag and standard opposite isotropic $2$-flag are each $\omega_h$-isotropic. 
	
	\begin{proof}[Proof of Theorem \ref*{thm:MaximalAntipodalSU}]
		We will show that for any $g \in U$, there exists $g'\in U'$ such that $gg'\tau_-$ is not antipodal to $\tau_-$. 
		To show that $gg'\tau_-$ is not antipodal to $\tau_-$, we only need to show that the $2 \times 2$ block $(gu)_{1n}$ has determinant $0$.
		Indeed, we may still apply \Cref{lem:transverseCondition} because the $\omega_h$-perpendicular of $\tau_{\{2\}}^{\rm opp}$ is equal to its $\omega$-perpendicular.
		
		As before, we write $(gg')_{1n}$ in the basis $I,R,T,P$.
		We first examine the $I$ term of $(gg')_{1n}$, where $g$ and $g'$ are expressed in the parameters described above.
		The coefficient of the $I$ term is given by
		\begin{align*}
			&-\frac12 \left( \abs{\alpha}^2+2\alpha \cdot u + \abs{u}^2 +\abs{\beta}^2-2\beta \cdot v +\abs{v}^2 -\abs{w}^2-\abs{z}^2\right)\\
			=&-\frac12 \left\lvert \alpha+ u\right\rvert^2 -\frac12 \left\lvert\beta-v\right\rvert^2 +\frac12 \abs{w}^2 + \frac12 \abs{z}^2
		\end{align*} 
		which vanishes for a suitable choice of $\alpha$ and $\beta$.
		
		It is then easy to choose $\gamma$ so that the $R$ term vanishes.
		This results in a $2\times 2$ block which may have zero or negative determinant. 
		If the determinant is zero, then we are done.
		Otherwise, the determinant is negative; but in this case, observe that by choosing sufficiently large $\gamma$ we can make the determinant positive. 
		By the intermediate value theorem, there then exists some $\gamma$ such that the determinant becomes zero.
	\end{proof}
	
	\begin{lemma}
		The limit set $\Lambda \cong S^{2n-3} \subset \mathcal{F}_{2,2n-2}$ associated to $\SU(n-1,1)\subset \SL(2n,\R)$ is not maximally antipodal.
	\end{lemma}
	
	\begin{proof}
		With the conventions as above, take 
		\[g=\begin{bmatrix} I & 0 & I \\ 0 & I_{2n-4} & 0 \\ 0 & 0 & I \end{bmatrix} .\]
		Then for any $g' \in U'$, $g \tau_-$ is antipodal to $g'\tau_-$. In fact the block $(g^{-1}g'\tau_{-})_{1n}$ is 
		\[-\frac{1}{2}\left(\abs{\alpha}^2 + \abs{\beta}^2 + 2\right)I + \gamma R\]
		Clearly, $g\tau_-$ is transverse to $\tau_+$. 
	\end{proof}
	
	\bibliographystyle{plain}
	
\end{document}